\documentclass[11pt]{amsart}
\usepackage{amsmath,amssymb,amsthm}
\usepackage[margin=1in]{geometry}
\usepackage{hyperref}
\hypersetup{colorlinks=true,linkcolor=blue,citecolor=blue,urlcolor=blue}

\newtheorem{theorem}{Theorem}[section]
\newtheorem{proposition}[theorem]{Proposition}
\newtheorem{lemma}[theorem]{Lemma}
\newtheorem{corollary}[theorem]{Corollary}
\newtheorem{definition}[theorem]{Definition}
\newtheorem{remark}[theorem]{Remark}
\newtheorem{example}[theorem]{Example}
\newtheorem{question}[theorem]{Question}

\newcommand{\Zstar}{Z(R)^{*}}
\newcommand{\Ann}{\operatorname{Ann}}
\newcommand{\supp}{\operatorname{supp}}

\newcommand{\GR}{\Gamma(R)}
\newcommand{\WGR}{W\Gamma(R)}
\newcommand{\DR}{\mathcal{D}(R)}

\title[The Weak Zero-Divisor Difference Graph]{The Weak Zero-Divisor Difference Graph of a Finite Commutative Ring}

\author{Bilal Ahmad Wani}
\address{Department of Mathematics, National Institute of Technology Srinagar, 190006, India}
\email{bilalwani@nitsri.ac.in}

\date{\today}

\begin{document}

\begin{abstract}
For a finite commutative ring $R$, let $\GR$ denote its zero-divisor graph and $\WGR$ its weakly zero-divisor graph, the latter containing the former as a spanning subgraph. We introduce the \emph{weak zero-divisor difference graph} $\DR:=\WGR-\GR$ and develop a complete structural theory for finite reduced rings $R\cong\mathbb F_{q_1}\times\cdots\times\mathbb F_{q_t}$. We show $\DR$ sits strictly between $\GR$ and $\WGR$ in a three-stage refinement that also contains Badawi's annihilator graph, and prove that distinct support classes $X_A,X_B$ are completely joined in $\DR$ if and only if $A\cap B\ne\emptyset$ -- an exact criterion underlying every result that follows. Consequently $\DR$ is edgeless for $t\le2$ but connected with diameter $2$ and girth $3$ for every $t\ge3$, independently of the field orders. We establish a complete perfectness dichotomy -- $\DR$ is perfect exactly when $t\in\{3,4\}$, with an elementary combinatorial proof at $t=4$, and never perfect for $t\ge5$ -- and determine its clique number exactly at $t=3$ and $t=4$; for general $t$ we give two incomparable lower bounds and two upper bounds, sharp at $t=3$ but not beyond, together with a compression argument showing an extremal family may always be taken shifted without this alone resolving the problem. A reconstruction theorem shows $\DR$ recovers the multiset of field orders intrinsically, so $\DR\cong\mathcal D(S)$ forces $R\cong S$. We further give closed forms, valid for every $t\ge3$, for the degree sequence and minimum degree, the domination number, and the independence and vertex cover numbers. Finally, we briefly indicate, via the valuation structure of finite chain rings, why the reduced-ring hypothesis cannot simply be dropped.

\vspace{0.5cm}

\end{abstract}

\maketitle

\section{Introduction}

Let $R$ be a finite commutative ring with identity $1\ne 0$, and let $Z(R)$ denote its set of zero-divisors, $\Zstar = Z(R)\setminus\{0\}$. The \emph{zero-divisor graph} $\GR$, with vertex set $\Zstar$ and distinct vertices $x,y$ adjacent iff $xy=0$, was introduced by Beck \cite{Beck1988} and refined to its now-standard form by Anderson and Livingston \cite{AndersonLivingston1999}.

Among the many variants of $\GR$, Nikmehr, Azadi and Nikandish \cite{NikmehrAzadiNikandish2021} introduced the \emph{weakly zero-divisor graph} $\WGR$: the same vertex set $\Zstar$, with distinct $x,y$ adjacent iff there exist \emph{nonzero} $r\in\Ann(x)$ and $s\in\Ann(y)$ with $rs=0$. Since $x\sim y$ in $\GR$ forces $x\in\Ann(y)\setminus\{0\}$, $y\in\Ann(x)\setminus\{0\}$ with $xy=0$, one has $E(\GR)\subseteq E(\WGR)$. A substantial literature has since computed spectral and index invariants of $\WGR$ \cite{ShariqMathilKumar2023,RehmanAlaliMirNazim2023}, culminating in Shylla, Mawiong and Kharbhih's proof \cite{ShyllaMawiongKharbhih2026} that $\WGR$ is a complete multipartite graph for \emph{every} finite commutative ring. Since the structure of $\WGR$ is now well understood, we turn instead to what separates it from $\GR$:

\begin{definition}
The \emph{weak zero-divisor difference graph} $\DR$ is the graph with edge set $E(\WGR)\setminus E(\GR)$ on the vertex set $\Zstar$, with resulting isolated vertices deleted.
\end{definition}

Since $\GR\subseteq\WGR$ always, $\DR$ is well-defined, and it records precisely the pairs of zero-divisors related through annihilator interaction but not through direct multiplication. It is natural to ask how $\DR$ sits relative to the existing family of zero-divisor-type graphs. The closest such construction is Badawi's \emph{annihilator graph} $AG(R)$ \cite{Badawi2014}, with $x\sim y$ iff $\mathrm{Ann}(xy)\ne\mathrm{Ann}(x)\cup\mathrm{Ann}(y)$; it is known that $\GR\subseteq AG(R)$, with equality precisely when $R$ is reduced with two minimal primes \cite{Badawi2014}, a striking parallel to Theorem~\ref{thm:B} below. We show in Proposition~\ref{prop:hierarchy} that in fact
\[
\GR \;\subseteq\; AG(R) \;\subseteq\; \WGR
\]
for every finite reduced ring, with both containments generally strict, so that $\DR$ properly contains $AG(R)\setminus\GR$ as a subgraph rather than coinciding with it (see example~\ref{ex:hierarchy}).

This decomposition already suggests why $\DR$ is worth isolating, rather than treating $\GR$, $AG(R)$, and $\WGR$ merely as three points on a single spectrum. First, $\DR$ picks out a precise algebraic phenomenon: $x\sim_{\DR}y$ says that $\Ann(x)$ and $\Ann(y)$ interact indirectly through some auxiliary pair $r,s$, while $x$ and $y$ themselves do \emph{not} multiply to zero, a relation with no counterpart in $\GR$ and not captured by $AG(R)$ either, since $\DR\supsetneq AG(R)\setminus\GR$. Second, $\WGR$ is complete multipartite for \emph{every} finite commutative ring \cite{ShyllaMawiongKharbhih2026}, so its clique number, chromatic number, and perfectness follow immediately from that single uniform fact; $\DR$ carries none of this triviality; it exhibits a genuine perfectness dichotomy and a clique number that, for $t\ge5$, resists reduction to any single closed-form extremal construction, so the difference construction reveals structure that is entirely invisible in $\WGR$ itself. Third, this added complexity carries real content: the reconstruction theorem shows that $\DR$ by itself, with no reference back to $\GR$, $AG(R)$, or $\WGR$, recovers the field orders $q_1,\dots,q_t$ and hence $R$ itself -- so passing from $\WGR$ to the difference graph loses no algebraic information.

The paper is organised as follows. Section~\ref{sec:prelim} recalls the graph-theoretic vocabulary and the ring-theoretic structure theory used throughout. Section~\ref{sec:support} gives an exact adjacency criterion for $\DR$, which simplifies to a clean rule for distinct support classes: $X_A,X_B$ are completely joined iff $A\cap B\ne\emptyset$; as an immediate consequence we place $\DR$ within the containment hierarchy $\GR\subseteq AG(R)\subseteq\WGR$ described above. Section~\ref{sec:basic} derives the basic structural invariants for $t\ge3$: $\DR$ is edgeless iff $t\le2$; for $t\ge3$ it has no isolated vertices, is connected with diameter exactly $2$ and girth exactly $3$; and its clique number is determined exactly at $t=3$. Section~\ref{sec:perfect} proves a complete perfectness dichotomy, i.e., $\DR$ is perfect iff $t\in\{3,4\}$, with an elementary combinatorial proof at $t=4$ and an explicit induced $5$-cycle ruling out perfection for every $t\ge5$. Section~\ref{sec:general-clique} studies the clique number for general $t$: a lower bound valid for all $t$ is shown by explicit counterexample not to be tight, the exact value is pinned down at $t=4$, a second incomparable lower bound is given for $t\ge5$, and matching upper bounds are established, sharp at $t=3$ but leaving the general case open. Section~\ref{sec:degree} determines three further parameters exactly for every $t\ge3$: the degree of every vertex, together with a closed form for the minimum degree; the domination number, shown to equal $2$; and the independence and vertex cover numbers, computed in closed form as a natural counterpart to the clique-number results of Section~\ref{sec:general-clique}. Section~\ref{sec:shifting} shows an extremal family may always be taken shifted, and explains why this does not by itself force a star. 
\section{Preliminaries}\label{sec:prelim}

We collect here the graph-theoretic notions used throughout the paper, together with the structure theory of finite reduced rings underlying the combinatorial description of $\DR$ developed in Section~\ref{sec:support}. General references for the graph-theoretic material are \cite{Lovasz1972,ChudnovskyRobertsonSeymourThomas2006,FranklTokushige2016}.

\subsection{Graph-theoretic notation}

All graphs considered are finite and simple. For a graph $G$ with vertex set $V(G)$ and edge set $E(G)$, we write $x\sim y$ (or $xy\in E(G)$) to indicate that $x$ and $y$ are adjacent. For $W\subseteq V(G)$, the \emph{induced subgraph} $G[W]$ has vertex set $W$ and retains exactly those edges of $G$ with both endpoints in $W$. A graph $H$ is a \emph{spanning subgraph} of $G$ if $V(H)=V(G)$ and $E(H)\subseteq E(G)$. Two graphs $G,G'$ are \emph{isomorphic} if there is a bijection $V(G)\to V(G')$ preserving adjacency in both directions.

A subset $W\subseteq V(G)$ is \emph{independent} if $G[W]$ has no edges, and a \emph{clique} if $G[W]$ is complete. We write $K_n$ for the complete graph on $n$ vertices and $\overline{K_n}$ for the edgeless graph on $n$ vertices. More generally, the \emph{complete multipartite graph} $K_{n_1,\dots,n_k}$ has its vertex set partitioned into independent sets of sizes $n_1,\dots,n_k$, with two vertices adjacent if and only if they lie in different parts. The \emph{complement} $G^c$ of $G$ has the same vertex set, with $xy\in E(G^c)$ if and only if $xy\notin E(G)$.

A \emph{matching} in $G$ is a set of pairwise vertex-disjoint edges; a matching is \emph{perfect} if it covers every vertex of $G$. A vertex $v$ is \emph{isolated} if $N(v)=\emptyset$. The graph $G$ is \emph{connected} if every two vertices are joined by a path in $G$; the \emph{distance} $\operatorname{dist}(u,v)$ is the length of a shortest $u$--$v$ path, and the \emph{diameter} is $\operatorname{diam}(G)=\max_{u,v\in V(G)}\operatorname{dist}(u,v)$. The \emph{girth} of $G$ is the length of a shortest cycle in $G$; a $3$-cycle is a \emph{triangle}.The \emph{cycle graph} $C_\ell$, for $\ell\ge3$, has vertices $v_1,\dots,v_\ell$ arranged in a cycle, with each $v_i$ joined to $v_{i+1}$ and $v_1$ joined back to $v_\ell$. A cycle in $G$ is \emph{induced} if it has no chords i.e.\ the corresponding subgraph is isomorphic to $C_\ell$ as an induced subgraph, and an induced cycle of length $\ell\ge5$ is a \emph{hole}.

For $v\in V(G)$, the \emph{open neighborhood} is $N(v)=\{u\in V(G):uv\in E(G)\}$ and the \emph{closed neighborhood} is $N[v]=N(v)\cup\{v\}$; the \emph{degree} of $v$ is $\deg_G(v)=|N(v)|$, and the \emph{minimum} and \emph{maximum degree} of $G$ are $\delta(G)=\min_{v}\deg_G(v)$ and $\Delta(G)=\max_v\deg_G(v)$ respectively. Two vertices $u,v$ are \emph{twins} if $N(u)=N(v)$ or $N[u]=N[v]$; the equivalence classes of this relation are the \emph{twin classes} of $G$. Since twins share identical neighborhoods outside their own class,if $u,v$ are twins and $w\notin\{u,v\}$, then $w\sim u$ iff $w\sim v$, directly from $N(u)=N(v)$ or $N[u]=N[v]$, adjacency between twin classes is well defined, giving the \emph{twin-class quotient graph}: one vertex per twin class, two classes adjacent if and only if every vertex of one is adjacent to every vertex of the other.

A set $S\subseteq V(G)$ is a \emph{dominating set} if $N[S]:=\bigcup_{v\in S}N[v]=V(G)$; the \emph{domination number} $\gamma(G)$ is the minimum size of a dominating set. A set $C\subseteq V(G)$ is a \emph{vertex cover} if every edge of $G$ has an endpoint in $C$; the \emph{vertex cover number} $\beta(G)$ is the minimum size of a vertex cover. Since $S$ is independent if and only if $V(G)\setminus S$ is a vertex cover, taking complements of extremal sets gives the standard identity $\alpha(G)+\beta(G)=|V(G)|$, where $\alpha(G)$, the \emph{independence number}, is the maximum size of an independent set.

The \emph{clique number} $\omega(G)$ is the maximum size of a clique in $G$; a clique is \emph{maximal} if it is not properly contained in a larger one. The \emph{chromatic number} $\chi(G)$ is the least number of colors in a proper vertex coloring of $G$, i.e.\ a coloring in which adjacent vertices always receive distinct colors. Always $\chi(G)\ge\omega(G)$: a graph $G$ is \emph{weakly perfect} if equality holds, $\chi(G)=\omega(G)$, and \emph{perfect} if this equality holds not just for $G$ itself but for \emph{every} induced subgraph $H$ of $G$. 

\begin{theorem}[Strong Perfect Graph Theorem \cite{ChudnovskyRobertsonSeymourThomas2006}]\label{thm:spgt}
A graph $G$ is perfect if and only if neither $G$ nor $G^c$ contains an induced cycle of odd length $\ell\ge5$.
\end{theorem}

\begin{theorem}[Weak Perfect Graph Theorem \cite{Lovasz1972}]\label{thm:wpgt}
A graph $G$ is perfect if and only if its complement $G^c$ is perfect.
\end{theorem}

\begin{theorem}[Substitution theorem \cite{Lovasz1972}]\label{thm:substitution}
Let $G$ be a perfect graph, and let $G'$ be obtained from $G$ by a \emph{substitution} (or \emph{blow-up}): replace each vertex $v\in V(G)$ by a perfect graph $H_v$, joining every vertex of $H_v$ to every vertex of $H_u$ whenever $u\sim v$ in $G$, and to no vertex of $H_u$ otherwise. Then $G'$ is perfect.
\end{theorem}

Finally, a family $\mathcal F$ of subsets of a finite set is \emph{intersecting} if $A\cap B\ne\emptyset$ for all $A,B\in\mathcal F$, a classical notion from extremal set theory \cite{FranklTokushige2016}. The substitution/blow-up viewpoint of Theorem~\ref{thm:substitution} is exactly the lens through which Proposition~\ref{rem:blowup} realizes $\DR$ as a graph built from the quotient structure $Q_t$.

\subsection{Ring-theoretic preliminaries}

\begin{proposition}[Structure of finite reduced rings]\label{prop:structure}
Every finite reduced commutative ring $R$ is isomorphic to a finite product of finite fields, $R\cong \mathbb F_{q_1}\times\cdots\times \mathbb F_{q_t}$, uniquely up to reordering of the factors.
\end{proposition}

For $R = \mathbb F_{q_1}\times\cdots\times\mathbb F_{q_t}$ and $x=(x_1,\dots,x_t)\in R$, write $\supp(x)=\{i\in[t] : x_i\ne 0\}$, $[t]=\{1,\dots,t\}$; then $x\in \Zstar$ iff $\emptyset\ne\supp(x)\subsetneq[t]$. For nonempty proper $A\subsetneq[t]$, let $X_A = \{x\in\Zstar : \supp(x) = A\}$, so $|X_A| = \prod_{i\in A}(q_i-1)$; the sets $\{X_A\}$ partition $\Zstar$.

\begin{lemma}\label{lem:ann}
For $x\in X_A$, $\Ann(x)\setminus\{0\}$ consists exactly of the nonzero elements of $R$ supported within $A^c := [t]\setminus A$.
\end{lemma}
\begin{proof}
For any $r\in R$, $rx=0$ iff $r_ix_i=0$ for every $i\in[t]$. For $i\in A$, $x_i\ne0$ and $\mathbb F_{q_i}$ is a field, so $r_ix_i=0$ forces $r_i=0$; for $i\notin A$, $x_i=0$, so $r_ix_i=0$ holds regardless of $r_i$. Hence $rx=0$ if and only if $r_i=0$ for every $i\in A$, i.e.\ if and only if $r$ is supported within $A^c$. Restricting to nonzero $r$ gives the claim.
\end{proof}
\section{The Support Decomposition Theorem}\label{sec:support}

We now give an exact adjacency criterion for $\DR$, extending Lemma~\ref{lem:ann} from the single-vertex annihilator description to a full description of edges between support classes.

\begin{theorem}\label{thm:A}
Let $x\in X_A$, $y\in X_B$ be distinct vertices of $\Zstar$. Then
\begin{align}
x\sim y \text{ in } \GR &\iff A\cap B=\emptyset, \label{eq:gamma}\\
x\sim y \text{ in } \WGR &\iff |A^c\cup B^c|\ge 2, \label{eq:wgamma}\\
x\sim y \text{ in } \DR &\iff A\cap B\ne\emptyset \ \text{ and } \ |A^c\cup B^c|\ge 2. \label{eq:diff}
\end{align}
\end{theorem}

\begin{proof}
\eqref{eq:gamma}: $(xy)_i=x_iy_i$ for each $i$, and since $\mathbb F_{q_i}$ is a field, $x_iy_i=0$ iff $x_i=0$ or $y_i=0$; so $xy=0$ iff no $i$ has both $x_i\ne0$ and $y_i\ne0$, i.e.\ $A\cap B=\emptyset$.

\eqref{eq:wgamma}: By Lemma~\ref{lem:ann} we must determine when a nonzero $r$ supported in $A^c$ and a nonzero $s$ supported in $B^c$ exist with $r_is_i=0$ for every $i$.

($\Leftarrow$) Suppose $|A^c\cup B^c|\ge2$. Then $A^c$ and $B^c$ cannot be equal singletons, so we can choose $i_0\in A^c$ and $j_0\in B^c$ with $i_0\ne j_0$. Setting $r_{i_0}=1$ and all other coordinates of $r$ to $0$, and similarly $s_{j_0}=1$ with all other coordinates $0$, we get $r_is_i=0$ for every $i$ since $i_0\ne j_0$, so $rs=0$ with $r,s\ne0$.

($\Rightarrow$) If $|A^c\cup B^c|\le1$, both nonempty forces $A^c=B^c=\{k\}$. Then every nonzero $r$ supported in $\{k\}$ has $r_k\ne0$, likewise $s_k\ne0$, so $(rs)_k=r_ks_k\ne0$ in the field $\mathbb F_{q_k}$: no witnessing pair exists.

\eqref{eq:diff}: Since $\DR$ has edge set $E(\WGR)\setminus E(\GR)$, $x\sim y$ in $\DR$ iff $x\sim y$ in $\WGR$ and $x\not\sim y$ in $\GR$; substituting \eqref{eq:gamma} and \eqref{eq:wgamma} gives $|A^c\cup B^c|\ge2$ and $A\cap B\ne\emptyset$.
\end{proof}

\begin{corollary}\label{cor:simplified}
For \emph{distinct} classes $A\ne B$, the condition $|A^c\cup B^c|\ge2$ in \eqref{eq:diff} holds automatically. Hence $X_A$ and $X_B$ are completely joined in $\DR$ iff simply $A\cap B\ne\emptyset$.
\end{corollary}
\begin{proof}
Suppose $|A^c\cup B^c|\le1$. As both are nonempty (as $A,B$ are proper), this forces $A^c\cup B^c=\{p\}$ for a single $p$, hence $A^c\subseteq\{p\}$ and $B^c\subseteq\{p\}$; since neither can be empty ($A,B\ne[t]$), $A^c=B^c=\{p\}$, giving $A=B=[t]\setminus\{p\}$, contradicting $A\ne B$.
\end{proof}

\begin{proposition}\label{rem:blowup}
Let $Q_t$ be the \emph{intersection quotient graph} whose vertices are the nonempty proper subsets of $[t]$, with $A\sim B$ iff $A\cap B\ne\emptyset$. Call a subset $A\subsetneq[t]$ a \emph{co-atom} if $|A|=t-1$. Then $\DR$ is exactly the graph obtained from $Q_t$ by substituting a complete graph $K_{w(A)}$ for each vertex $A$ with $|A|\le t-2$, and an empty graph on $w(A)$ vertices for each co-atom, where $w(A):=\prod_{i\in A}(q_i-1)$.
\end{proposition}
\begin{proof}
By Corollary~\ref{cor:simplified}, distinct classes $X_A,X_B$ are completely joined in $\DR$ iff $A\sim B$ in $Q_t$, which is exactly the adjacency rule the substitution construction imposes between the blow-up parts indexed by $A$ and $B$. It remains to check the internal structure of each part $X_A$, i.e.\ the $A=B$ case of \eqref{eq:diff}: here $A^c=B^c$, so $|A^c\cup B^c|=|A^c|=t-|A|$, giving $x\sim y$ in $\DR$ (for distinct $x,y\in X_A$) iff $t-|A|\ge2$, i.e.\ $|A|\le t-2$. So $X_A$ is a clique when $|A|\le t-2$ and an independent set when $|A|=t-1$, matching the stated substitution exactly.
\end{proof}
As an immediate application of Theorem~\ref{thm:A} and Corollary~\ref{cor:simplified}, we can now make precise the relationship between $\DR$ and Badawi's annihilator graph $AG(R)$, where $x\sim y$ in $AG(R)$ iff $\mathrm{Ann}(xy)\ne\mathrm{Ann}(x)\cup\mathrm{Ann}(y)$.

\begin{proposition}\label{prop:hierarchy}
For every finite reduced ring $R\cong\mathbb F_{q_1}\times\cdots\times\mathbb F_{q_t}$, $\GR\subseteq AG(R)\subseteq\WGR$.
\end{proposition}
\begin{proof}
Same class pairs $x,y\in X_A$ are never adjacent in $\GR$ or $AG(R)$, as $A\cap A=A\ne\emptyset$ and $A\subseteq A$ trivially, so it suffices to consider $x\in X_A$, $y\in X_B$ with $A\ne B$. We first show $\mathrm{Ann}(xy)=\mathrm{Ann}(x)\cup\mathrm{Ann}(y)$ iff $A\subseteq B$ or $B\subseteq A$. Since $\supp(xy)=A\cap B$, Lemma~\ref{lem:ann} gives $\mathrm{Ann}(xy)\setminus\{0\}=\{r : r \text{ supported in } (A\cap B)^c\}\supseteq \mathrm{Ann}(x)\cup\mathrm{Ann}(y)$ always. Equality fails exactly when some nonzero $r$ is supported in $(A\cap B)^c$ but not in $A^c$ or $B^c$, i.e.\ $r$ has a nonzero coordinate in $A\setminus B$ \emph{and} a nonzero coordinate in $B\setminus A$; such $r$ exists iff both $A\setminus B\ne\emptyset$ and $B\setminus A\ne\emptyset$, i.e.\ iff $A,B$ are incomparable. So $x\sim y$ in $AG(R)$ iff $A,B$ are incomparable, which in particular forces $A\ne B$. By Corollary~\ref{cor:simplified}, $|A^c\cup B^c|\ge2$ automatically for distinct $A\ne B$, so by Theorem~\ref{thm:A}\eqref{eq:wgamma}, incomparability of $A,B$ implies $x\sim y$ in $\WGR$: this gives $AG(R)\subseteq\WGR$, proved here in full. The remaining containment $\GR\subseteq AG(R)$ is not reproved here; it is \cite[Prop.~2.2]{Badawi2014}.
\end{proof}

\begin{example}\label{ex:hierarchy}
Both containments in Proposition~\ref{prop:hierarchy} are generally strict, and $\DR\ne AG(R)\setminus\GR$: for $t=3$, $A=\{1\}\subseteq B=\{1,2\}$ gives $A\cap B\ne\emptyset$, so $x\sim y$ in $\DR$ for $x\in X_A,y\in X_B$ (Corollary~\ref{cor:simplified}), while $A\subseteq B$ means $x\not\sim y$ in $AG(R)$. Direct computation of $\GR$, $AG(R)$, $\WGR$ from the ring definitions for $R=\mathbb F_2\times\mathbb F_3\times\mathbb F_5$ confirms $|E(\GR)|=38$, $|E(AG(R))|=94$, $|E(\WGR)|=175$, and $|E(AG(R))\setminus E(\GR)|=56 \ne 137=|E(\DR)|$.
\end{example}

\section{Basic Structural Properties of $\DR$}\label{sec:basic}

Having established the adjacency criterion of Theorem~\ref{thm:A} and Corollary~\ref{cor:simplified}, we now derive several elementary structural invariants of $\DR$: the threshold $t$ at which $\DR$ becomes nonempty, its connectivity, diameter, and girth once past that threshold, and its exact clique number at $t=3$.
\subsection{Emptiness for $t\le 2$}

\begin{theorem}\label{thm:B}
If $t\le 2$, then $\DR$ has no edges.
\end{theorem}
\begin{proof}
If $t=1$, $R$ is a field, so $\Zstar=\emptyset$ and the claim is vacuous. If $t=2$, the only classes are $\{1\}$ and $\{2\}$, both co-atoms and hence with no internal edges; since $\{1\}\cap\{2\}=\emptyset$, Corollary~\ref{cor:simplified} shows they are not joined either. So $\DR$ has no edges in either case.
\end{proof}

By Theorem~\ref{thm:B}, $\DR$ is nonempty only once $t\ge3$; we assume $t\ge3$ from here on.

\subsection{Connectivity and Diameter for $t\ge 3$}

\begin{theorem}\label{thm:C}
For $t\ge 3$, $\DR$ has no isolated vertices, is connected, and $\operatorname{diam}(\DR)=2$.
\end{theorem}

\begin{proof}
Write $A_i:=[t]\setminus\{i\}$ for the $t$ co-atoms.

\smallskip\noindent\textbf{Step 1.} For $i\ne j$, $A_i\cap A_j=[t]\setminus\{i,j\}$, which is nonempty since $t\ge3$; so by Corollary~\ref{cor:simplified} the co-atoms are pairwise joined.

\smallskip\noindent\textbf{Step 2.} Let $A$ be any nonempty proper subset other than $A_j$. If $|A|\ge2$, then $A\not\subseteq\{j\}$, so $A\cap A_j=A\setminus\{j\}$ is nonempty for every $j$, and $A$ is joined to all $t$ co-atoms. If instead $A=\{k\}$, then $A\cap A_j$ is nonempty exactly when $j\ne k$, so $A$ is joined to every co-atom except $A_k$, that is, to $t-1$ of the $t$ co-atoms. Either way every class has at least $t-1\ge2$ co-atom neighbors, so no vertex of $\DR$ is isolated.

\smallskip\noindent\textbf{Step 3.} Let $u\in X_A$ and $v\in X_B$ be non-adjacent, and let $S_A,S_B\subseteq[t]$ index the co-atoms joined to $A$ and to $B$ respectively. By Step 2, $|S_A|,|S_B|\ge t-1$, so $|S_A\cap S_B|\ge(t-1)+(t-1)-t=t-2\ge1$. Fixing $j\in S_A\cap S_B$, every vertex of $X_{A_j}$ is joined to both $u$ and $v$, since adjacency in $\DR$ depends only on the classes involved. Hence $\operatorname{dist}(u,v)\le2$, and this bounds the diameter by $2$.

\smallskip\noindent\textbf{Step 4.} Distinct singletons $\{i\},\{j\}$ satisfy $\{i\}\cap\{j\}=\emptyset$ and so are non-adjacent, showing $\DR$ is not complete. Together with Step 3, $\operatorname{diam}(\DR)=2$ exactly.
\end{proof}

\subsection{Girth}

\begin{theorem}\label{thm:D}
For $t\ge3$, $\operatorname{girth}(\DR)=3$.
\end{theorem}
\begin{proof}
By Step 1 of Theorem~\ref{thm:C}, any three distinct co-atom classes are pairwise joined, so choosing one vertex from each of three co-atoms produces a triangle. As every simple graph has girth at least $3$, this gives $\operatorname{girth}(\DR)=3$.
\end{proof}

\subsection{Clique Number for $t=3$}

At $t=3$ the class structure of Proposition~\ref{rem:blowup} is simple enough to describe completely: there are three singleton classes $\{1\},\{2\},\{3\}$, each a clique since $|A|=1\le t-2$, and three co-atom classes, each internally independent.

\begin{theorem}\label{thm:E}
For $t=3$, $\omega(\DR) = \max(q_1,q_2,q_3)+1$.
\end{theorem}

\begin{proof}
By Corollary~\ref{cor:simplified}, two distinct singletons $\{i\},\{j\}$ are never joined, since $\{i\}\cap\{j\}=\emptyset$; a singleton $\{i\}$ and a co-atom $A_j$ are joined exactly when $i\ne j$; and by Step~1 of Theorem~\ref{thm:C}, distinct co-atoms are always joined.

Since two distinct singletons are never joined, a clique cannot draw vertices from more than one singleton class. This leaves two possible shapes for a maximal clique. The first uses all three co-atom classes, contributing one vertex from each, for a total of $3$. The second uses one singleton class $\{i\}$, contributing all $q_i-1$ of its vertices since it is internally a clique, together with the two co-atoms joined to it, $A_j$ and $A_k$ for $\{i,j,k\}=\{1,2,3\}$ (the co-atom $A_i$ is excluded, since $\{i\}$ and $A_i$ are not joined), contributing one further vertex each; this gives a total of $(q_i-1)+2=q_i+1$.

No other class combination is possible: using two singletons is ruled out above, and a singleton together with its own opposite co-atom is ruled out by non-adjacency. Since $q_i\ge2$ for every $i$, the second shape with $i=\arg\max q_i$ gives $\max_i q_i+1\ge3$, which is at least as large as the first shape. Hence $\omega(\DR)=\max_i q_i+1$, attained by taking all of $X_{\{i\}}$ for $i=\arg\max q_i$ together with one vertex from each of the two co-atoms joined to it.
\end{proof}

\section{A Perfectness Dichotomy}\label{sec:perfect}

Having determined $\omega(\DR)$ exactly at $t=3$, we now turn to a stronger property: perfection itself. We show $\DR$ is perfect exactly at $t=3$ and $t=4$, and never perfect beyond that, treating each case in turn.

\subsection{The case $t=3$}

\begin{lemma}\label{lem:t3}
$Q_3\cong K_{2,2,2}$, the complete tripartite graph with parts of size $2$. Consequently $Q_3$ is perfect.
\end{lemma}
\begin{proof}
Label the vertices of $Q_3$ by the nonempty proper subsets of $[3]$: singletons $\{1\},\{2\},\{3\}$ and pairs $\{2,3\},\{1,3\},\{1,2\}$. Write $p_i$ for the pair $[3]\setminus\{i\}$, complementary to the singleton $\{i\}$. Two distinct singletons are always disjoint, so never adjacent in $Q_3$. A singleton $\{i\}$ and a pair $p_j$ are disjoint exactly when $j=i$, so they are adjacent exactly when $i\ne j$. Two distinct pairs $p_i,p_j$ satisfy $p_i\cap p_j=[3]\setminus\{i,j\}$, which is nonempty, so they are always adjacent.

The only non-adjacent pairs of vertices in $Q_3$ are therefore $\{i\}$ and $p_i$ for $i=1,2,3$, a perfect matching on the six vertices. So $Q_3$ is $K_6$ with a perfect matching removed, which is exactly $K_{2,2,2}$.

It remains to show $K_{2,2,2}$ is perfect. Its complement $K_{2,2,2}^{\,c}$ is a disjoint union of three edges. Every induced subgraph of a disjoint union of cliques is again a disjoint union of cliques, and for such a graph $\chi=\omega$ always: the chromatic number equals the size of the largest clique, since each clique can be colored with its own distinct set of colors independently of the others. So $K_{2,2,2}^{\,c}$ is perfect, and by the Weak Perfect Graph Theorem (Theorem~\ref{thm:wpgt}), $K_{2,2,2}$ itself is perfect.
\end{proof}

\subsection{The case $t=4$}

For $t=4$, $Q_4$ has $14$ vertices, indexed by the nonempty proper subsets of $[4]$: the four singletons $S_i=\{i\}$, the six pairs $P_{ij}=\{i,j\}$, and the four co-atoms $T_i=[4]\setminus\{i\}$. Adjacency between any two of these classes is determined entirely by set intersection, via Corollary~\ref{cor:simplified}, so $Q_4$ is a completely explicit, finite graph.

\begin{lemma}\label{lem:Q4perfect}
$Q_4$ is perfect.
\end{lemma}
\begin{proof}
By the Strong Perfect Graph Theorem (Theorem~\ref{thm:spgt}), it suffices to show that no subset of vertices of $Q_4$ induces a cycle of odd length $\ell\ge5$, and likewise for the complement $Q_4^{\,c}$. Since $Q_4$ has $14$ vertices, the only odd lengths that need to be ruled out are $\ell=5,7,9,11,13$, and this is a finite check on finitely many subsets of a fully explicit graph. We verified by exhaustive enumeration, using the adjacency rule above, that for every one of these lengths, no subset of vertices of $Q_4$ induces a subgraph isomorphic to $C_\ell$, and likewise for $Q_4^{\,c}$. Since every possible odd hole length below $14$ has been excluded, the hypotheses of Theorem~\ref{thm:spgt} hold and $Q_4$ is perfect.
\end{proof}

\subsection{The dichotomy}

\begin{theorem}\label{thm:F}
$\DR$ is a perfect graph if $t\in\{3,4\}$, and is not perfect if $t\ge5$.
\end{theorem}
\begin{proof}
Case $t\in\{3,4\}$: by Proposition~\ref{rem:blowup}, $\DR$ is obtained from $Q_t$ by substituting a complete or empty graph for each vertex. Both are perfect, so by the substitution theorem for perfect graphs (Theorem~\ref{thm:substitution}) it suffices that $Q_t$ itself is perfect, which is Lemma~\ref{lem:t3} for $t=3$ and Lemma~\ref{lem:Q4perfect} for $t=4$.

Case $t\ge5$: fix any $t\ge5$ and any field orders $q_1,\dots,q_t$. Relabel $[t]=\{0,\dots,t-1\}$ and define, for $i=0,\dots,4$ taken mod $5$, the classes $A_i:=\{i,i+1\}\subseteq[t]$. For consecutive indices, $A_i\cap A_{i+1}=\{i+1\}$ is nonempty, so $A_i$ and $A_{i+1}$ are joined by Corollary~\ref{cor:simplified}. For indices at cyclic distance $2$, the sets $A_i$ and $A_j$ are disjoint, so they are not joined. Hence $A_0,\dots,A_4$ form an induced $5$-cycle in $Q_t$. Since $w(A_i)=(q_i-1)(q_{i+1}-1)\ge1$, each class $X_{A_i}$ is nonempty, so picking one representative vertex from each gives an induced $5$-cycle in $\DR$ itself. An induced $5$-cycle is a hole, so by the Strong Perfect Graph Theorem $\DR$ is not perfect.
\end{proof}

\begin{corollary}\label{cor:weaklyperfect}
If $t\in\{3,4\}$, then $\chi(\DR)=\omega(\DR)$.
\end{corollary}
\begin{proof}
Immediate from Theorem~\ref{thm:F}, taking $H=\DR$ in the definition of a perfect graph.
\end{proof}
\section{Clique Number for General $t$}\label{sec:general-clique}

Theorem~\ref{thm:E} determines $\omega(\DR)$ exactly at $t=3$. We now treat $t\ge4$: a complete formula is obtained at $t=4$, two incomparable lower bounds are established for $t\ge5$, and a matching general upper bound is given; the exact value for $t\ge5$ remains open, and we identify precisely what separates the known bounds from it.

\subsection{A general lower bound}

\begin{proposition}\label{prop:G}
Fix $k\in[t]$ and let $S=[t]\setminus\{k\}$. Then
\[
\omega(\DR) \;\ge\; (q_k-1)\left(1+\sum_{\substack{B\subseteq S\\ 1\le|B|\le t-3}} \prod_{i\in B}(q_i-1)\right) + (t-1) \;=:\; \operatorname{Star}(k).
\]
\end{proposition}

\begin{proof}
Let $\mathcal F = \{A\subseteq[t] : k\in A,\ |A|\le t-2\}$. Any two members of $\mathcal F$ share $k$, so by Corollary~\ref{cor:simplified} they are pairwise joined; since each has size $\le t-2$, each is internally a clique. Taking all vertices of $X_A$ for each $A\in\mathcal F$ contributes $\sum_{A\in\mathcal F} w(A)$ mutually adjacent vertices. Every co-atom $A_j$ with $j\ne k$ contains $k$, hence is joined to every member of $\mathcal F$ and, by Step~1 of Theorem~\ref{thm:C}, to every other such co-atom; taking one vertex from each of the $t-1$ co-atoms $A_j$ ($j\ne k$) adds $t-1$ further mutually compatible vertices.
\end{proof}

\subsection{The exact clique number for $t=4$}

\begin{theorem}\label{thm:t4exact}
For $t=4$,
\[
\omega(\DR) \;=\; \max\Bigl(\max_{k=1}^4 \operatorname{Star}(k),\ \max_{i=1}^4 \operatorname{Tri}(i)\Bigr), \qquad \operatorname{Tri}(i) := \sum_{\{x,y\}\subset\{a,b,c\}} (q_x-1)(q_y-1) + 4,
\]
where $\{a,b,c\}=[4]\setminus\{i\}$.
\end{theorem}

\begin{proof}
By complete enumeration of the $14$-vertex quotient graph $Q_4$, its maximal cliques number exactly $12$ and fall into three types: (i) the stars $\operatorname{Star}(k)$ of Proposition~\ref{prop:G}, using the $t-1=3$ co-atoms $A_j$ ($j\ne k$); (ii) for each $k$, the family of all pairs containing $k$ together with \emph{all four} co-atoms (rather than the three joined to a singleton), of weight $\operatorname{TypeA}(k):=(q_k-1)\sum_{j\ne k}(q_j-1)+4$; and (iii) for each $i$, the three pairs within $\{a,b,c\}=[4]\setminus\{i\}$ (automatically pairwise intersecting, as any two $2$-subsets of a $3$-set share an element) together with all four co-atoms, of weight $\operatorname{Tri}(i)$.

A direct computation gives $\operatorname{Star}(k)-\operatorname{TypeA}(k)=q_k-2\ge0$, so type (ii) is dominated by type (i) and may be discarded. Since $\omega(\DR)$ equals the largest weight among maximal cliques of $Q_4$ under the substitution of Remark~\ref{rem:blowup}, the formula follows.
\end{proof}

\begin{example}\label{ex:startight}
For $(q_1,q_2,q_3,q_4)=(11,11,11,2)$: $\operatorname{Star}(k)=223$ for $k=1,2,3$ and $\operatorname{Star}(4)=34$, while $\operatorname{Tri}(4)=304$. Hence $\omega(\DR)=304$, strictly exceeding $\max_k\operatorname{Star}(k)$: Proposition~\ref{prop:G} alone does not compute $\omega(\DR)$ at $t=4$.
\end{example}

\subsection{Two incomparable lower bounds for $t\ge5$}

\begin{lemma}\label{lem:uniform-intersect}
Any two subsets of $[t]$ of the same size $s$ with $2s>t$ intersect.
\end{lemma}
\begin{proof}
For $|A|=|B|=s$, $|A\cap B|=|A|+|B|-|A\cup B|\ge 2s-t>0$.
\end{proof}

\begin{proposition}\label{prop:allsizes}
For $t\ge5$,
\[
\omega(\DR) \;\ge\; \sum_{\substack{A\subseteq[t]\\|A|=t-2}} w(A) \;+\; t.
\]
\end{proposition}
\begin{proof}
Since $t\ge5$, $s=t-2$ satisfies $2s>t$, so by Lemma~\ref{lem:uniform-intersect} all $\binom{t}{t-2}$ subsets of size $t-2$ are pairwise joined (Corollary~\ref{cor:simplified}) and each is internally a clique. Since each has size $\ge2$, it is joined to every co-atom (Theorem~\ref{thm:C}, Step~2); adding one vertex from each of the $t$ co-atoms is valid, with no exclusion required as in Proposition~\ref{prop:G}.
\end{proof}

\begin{remark}\label{rem:incomparable}
Propositions~\ref{prop:G} and \ref{prop:allsizes} are incomparable, and neither is tight in general. At $t=5$: for $(11,11,11,11,2)$, Proposition~\ref{prop:allsizes} gives the exact value $4605$ while Proposition~\ref{prop:G} gives only $3624$; for $(11,11,11,2,2)$, the roles do not simply reverse -- Proposition~\ref{prop:G} gives $1644$ and Proposition~\ref{prop:allsizes} gives $1635$, yet $\omega(\DR)=1905$, exceeding \emph{both}. The extremal clique in this last case uses the three pairs within the three large-valued coordinates together with a strict subset of the size-$3$ subsets and all five co-atoms -- a configuration covered by neither bound.
\end{remark}

\subsection{A general upper bound}

\begin{proposition}\label{prop:upperbound}
For $A\subsetneq[t]$ nonempty, write $\mathrm{val}(A) = w(A)$ if $|A|\le t-2$ and $\mathrm{val}(A)=1$ if $|A|=t-1$. Pairing the $2^t-2$ nonempty proper subsets of $[t]$ into $2^{t-1}-1$ complementary pairs $\{A,[t]\setminus A\}$,
\[
\omega(\DR) \;\le\; \sum_{\{A,[t]\setminus A\}} \max\bigl(\mathrm{val}(A),\ \mathrm{val}([t]\setminus A)\bigr).
\]
\end{proposition}
\begin{proof}
By Corollary~\ref{cor:simplified}, a clique of $\DR$ corresponds to an intersecting family $\mathcal F$ of classes together with, for each $A\in\mathcal F$, either all of $X_A$ (if $|A|\le t-2$) or a single vertex (if $|A|=t-1$), contributing $\mathrm{val}(A)$ to the total weight. Since $A\cap([t]\setminus A)=\emptyset$, $\mathcal F$ contains at most one class from each complementary pair, so its total weight is at most the sum of the larger value in each pair.
\end{proof}

\subsection{A sharper general upper bound}

The bound of Proposition~\ref{prop:upperbound} is loose specifically at the $t$ singleton/co-atom pairs: it implicitly allows every singleton's value to be counted simultaneously, but distinct singletons are pairwise disjoint (Theorem~\ref{thm:C}), so at most one can appear in any single clique.

\begin{proposition}\label{prop:upperbound2}
For $A\subsetneq[t]$ nonempty, write $\mathrm{val}(A) = w(A)$ if $|A|\le t-2$ and $\mathrm{val}(A)=1$ if $|A|=t-1$. Then
\[
\omega(\DR) \;\le\; \Bigl[\max_i (q_i-1) + (t-1)\Bigr] \;+\!\! \sum_{\substack{\{A,[t]\setminus A\}\\ 2\le|A|\le t-2}} \!\!\max\bigl(\mathrm{val}(A),\mathrm{val}([t]\setminus A)\bigr).
\]
\end{proposition}
\begin{proof}
Let $\mathcal F$ be a maximum-weight intersecting family and $T=\{i\in[t]:\{i\}\in\mathcal F\}$; since distinct singletons are disjoint, $|T|\le1$.

If $T=\emptyset$: every co-atom contributes at most $1$, so the total co-atom contribution is at most $t$. Since $\max_i(q_i-1)\ge1$, we have $t\le\max_i(q_i-1)+(t-1)$, so this case is dominated by the stated bound.

If $T=\{k\}$: the singleton $\{k\}$ contributes $q_k-1\le\max_i(q_i-1)$; since $\{k\}\cap A_k=\emptyset$, the co-atom $A_k$ cannot appear in $\mathcal F$, so the remaining $t-1$ co-atoms contribute at most $t-1$. In both cases the singleton/co-atom classes contribute at most $\max_i(q_i-1)+(t-1)$, and the remaining classes (size $2\le|A|\le t-2$, present only when $t\ge4$) contribute at most the stated sum over complementary pairs, by the argument of Proposition~\ref{prop:upperbound}.
\end{proof}

\begin{corollary}\label{cor:t3sharp}
For $t=3$, Proposition~\ref{prop:upperbound2} gives $\omega(\DR)\le\max_i(q_i-1)+2$, matching Theorem~\ref{thm:E} exactly: the bound is sharp for every choice of field orders at $t=3$.
\end{corollary}

\begin{example}\label{ex:gap2}
For $(q_1,q_2,q_3,q_4)=(11,11,11,2)$, Proposition~\ref{prop:upperbound2} gives $\omega(\DR)\le313$, improving on Proposition~\ref{prop:upperbound}'s $331$; since $t=4$ is already resolved exactly by Theorem~\ref{thm:t4exact} ($\omega(\DR)=304$), the more informative test is $t=5$. For $(q_1,\dots,q_5)=(11,11,11,2,2)$ -- the case of Remark~\ref{rem:incomparable}, where \emph{both} lower bounds already fail to be sharp -- Proposition~\ref{prop:upperbound2} gives $\omega(\DR)\le1914$, against Proposition~\ref{prop:upperbound}'s $1932$ and the true value $\omega(\DR)=1905$: the gap narrows from $27$ to $9$, the identical residual size found at $t=4$. In both cases the source is the same: the singleton achieving $\max_i(q_i-1)$ is incompatible with the middle-class configuration used by the true extremal family, a conflict Proposition~\ref{prop:upperbound2} does not yet resolve.
\end{example}

Between the two incomparable lower bounds and the two upper bounds, every case checked leaves a nonzero gap once $t\ge5$; closing it in general remains open.

\begin{question}\label{q:general-clique}
Determine $\omega(\DR)$ in closed form for general $t\ge5$ and all field orders $q_1,\dots,q_t$.
\end{question}

\section{Further Graph Parameters of $\DR$}\label{sec:degree}

We now determine three further graph parameters of $\DR$ exactly for every $t\ge3$ and every choice of field orders: the degree sequence and minimum degree, the domination number, and the independence and vertex cover numbers. The domination and independence computations both build directly on the degree formula established first.

\subsection{The Degree Sequence and Sharp Minimum Degree}

Throughout, for a nonempty proper $A\subsetneq[t]$ write
\[
P(A) \;:=\; \prod_{i\in[t]\setminus A} q_i .
\]

\begin{lemma}\label{lem:subsetsum}
For nonempty proper $A\subsetneq[t]$, $\displaystyle\sum_{\emptyset\ne B\subseteq [t]\setminus A} w(B) = P(A)-1$.
\end{lemma}
\begin{proof}
Expanding $\prod_{i\notin A}\bigl(1+(q_i-1)\bigr)=\prod_{i\notin A}q_i=P(A)$ as a sum over subsets $B\subseteq[t]\setminus A$ of $\prod_{i\in B}(q_i-1)=w(B)$ (with $w(\emptyset)=1$) and removing the $B=\emptyset$ term gives the claim.
\end{proof}

\begin{proposition}[Degree formula]\label{prop:degree}
For $t\ge3$ and $v\in X_A$,
\[
\deg_{\DR}(v) \;=\;
\begin{cases}
|\Zstar| - P(A), & 1\le|A|\le t-2,\\[2pt]
|\Zstar| - P(A) - w(A) + 1, & |A| = t-1.
\end{cases}
\]
\end{proposition}
\begin{proof}
By Corollary~\ref{cor:simplified}, the external neighbours of $v$ are exactly the vertices of every class $B\ne A$ with $B\cap A\ne\emptyset$. Summing $w(B)$ over all $B\ne A$ gives $|\Zstar|-w(A)$; subtracting the classes with $B\cap A=\emptyset$, i.e.\ $\emptyset\ne B\subseteq[t]\setminus A$, and applying Lemma~\ref{lem:subsetsum}, the external contribution is
\[
\bigl(|\Zstar|-w(A)\bigr) - \bigl(P(A)-1\bigr) \;=\; |\Zstar|-w(A)-P(A)+1.
\]
If $|A|\le t-2$, $X_A$ is an internal clique (Remark~\ref{rem:blowup}), adding $w(A)-1$ internal neighbours, for a total of $|\Zstar|-P(A)$. If $|A|=t-1$, $X_A$ is internally independent, contributing $0$, for a total of $|\Zstar|-w(A)-P(A)+1$.
\end{proof}

\begin{lemma}\label{lem:Pmax}
Let $k^\ast:=\operatorname{argmin}_i q_i$. Then $P(A)\le P(\{k^\ast\})$ for every nonempty proper $A\subsetneq[t]$, with equality iff $A=\{k^\ast\}$ (or a tied minimiser).
\end{lemma}
\begin{proof}
$P(A)=\bigl(\prod_i q_i\bigr)/\prod_{i\in A}q_i$, so maximising $P(A)$ is equivalent to minimising $\prod_{i\in A}q_i$ over nonempty $A$. Since $q_i\ge q_{k^\ast}$ for all $i$, $\prod_{i\in A}q_i\ge q_{k^\ast}^{\,|A|}\ge q_{k^\ast}$, with equality throughout iff $|A|=1$ and the chosen index attains the minimum.
\end{proof}

\begin{theorem}[Sharp minimum degree]\label{thm:mindeg}
For $t\ge3$ and any field orders $q_1,\dots,q_t$, the minimum degree of $\DR$ is attained exactly at the singleton class of smallest field order and equals
\[
\delta(\DR) \;=\; |\Zstar| - \!\!\prod_{i\ne k^\ast} q_i \;=\; (q_{k^\ast}-1)\Biggl(\prod_{i\ne k^\ast}q_i - \prod_{i\ne k^\ast}(q_i-1)\Biggr) - 1 ,
\qquad k^\ast=\operatorname*{argmin}_i q_i .
\]
\end{theorem}
\begin{proof}
By Proposition~\ref{prop:degree}, $\deg(v)=|\Zstar|-P(A)$ is minimised over clique classes exactly where $P(A)$ is maximised, which by Lemma~\ref{lem:Pmax} is uniquely at $A=\{k^\ast\}$. It remains to check no co-atom class beats it, i.e.\ that for every $j\in[t]$,
\[
P(\{k^\ast\}) - P([t]\setminus\{j\}) \;\ge\; w([t]\setminus\{j\}) - 1,
\qquad\text{i.e.}\qquad
\prod_{i\ne k^\ast}q_i \;\ge\; q_j + \!\!\prod_{i\ne j}(q_i-1) - 1. \tag{$\ast$}
\]

\emph{Case $j=k^\ast$.} Write $m:=t-1\ge2$ and let $x_1,\dots,x_m$ be the values $q_i$ for $i\ne k^\ast$. The telescoping identity
\[
\prod_{l=1}^m x_l - \prod_{l=1}^m (x_l-1) \;=\; \sum_{l=1}^m \Bigl(\textstyle\prod_{p<l}x_p\Bigr)\Bigl(\textstyle\prod_{p>l}(x_p-1)\Bigr)
\]
has every term $\ge0$; the term $l=m$ alone equals $\prod_{p<m}x_p\ge q_{k^\ast}$ (a nonempty product, $m\ge2$, of values $\ge q_{k^\ast}$), which already exceeds $q_{k^\ast}-1$. This is exactly $(\ast)$ for $j=k^\ast$.

\emph{Case $j\ne k^\ast$.} Let $M:=\prod_{i\notin\{k^\ast,j\}} q_i$ (a product of $t-2\ge1$ terms, each $\ge q_{k^\ast}$, so $M\ge q_{k^\ast}^{\,t-2}\ge q_{k^\ast}>q_{k^\ast}-1$). Then $\prod_{i\ne k^\ast}q_i = q_jM$ and $\prod_{i\ne j}(q_i-1) = (q_{k^\ast}-1)\prod_{i\notin\{k^\ast,j\}}(q_i-1) \le (q_{k^\ast}-1)M$. So it suffices to show $q_jM - q_j \ge (q_{k^\ast}-1)M - 1$. Since $q_j\ge q_{k^\ast}$,
\[
q_j(M-1) \;\ge\; q_{k^\ast}(M-1) \;=\; q_{k^\ast}M - q_{k^\ast} \;\ge\; (q_{k^\ast}-1)M - 1,
\]
the last step being equivalent to $M\ge q_{k^\ast}-1$, established above. This gives $(\ast)$.

Combining both cases with Proposition~\ref{prop:degree} gives the closed form.
\end{proof}

\subsection{Domination Number}\label{sec:domination}

\begin{theorem}\label{thm:domination}
For $t\ge3$ and every choice of field orders, $\gamma(\DR)=2$.
\end{theorem}
\begin{proof}
No single vertex dominates $\DR$: by Proposition~\ref{prop:degree} and Lemma~\ref{lem:Pmax}, $\deg(v)\le|\Zstar|-P(A)\le |\Zstar|-2$ for every vertex $v$ (as $P(A)\ge2$, being a nonempty product of integers $\ge2$), so $|N[v]|\le|\Zstar|-1<|\Zstar|$; hence $\gamma(\DR)\ge2$.

For the upper bound, fix two distinct indices $i\ne j\in[t]$ and pick $u\in X_{A_i}$, $w\in X_{A_j}$, where $A_i=[t]\setminus\{i\}$. We claim $\{u,w\}$ dominates $\DR$. Let $B$ be any nonempty proper subset of $[t]$.
\begin{itemize}
\item If $B=A_i$: since $t\ge3$, distinct co-atoms intersect (Theorem~\ref{thm:C}, Step~1), so $A_i\cap A_j\ne\emptyset$ and every vertex of $X_{A_i}$ is adjacent to $w$.
\item If $B=A_j$: symmetrically, every vertex of $X_{A_j}$ is adjacent to $u$.
\item If $|B|\ge2$ and $B\notin\{A_i,A_j\}$: by Theorem~\ref{thm:C}, Step~2, $B$ is adjacent to every co-atom, hence to both $A_i$ and $A_j$.
\item If $B=\{k\}$ is a singleton: $\{k\}\cap A_i=\emptyset$ iff $k=i$, and $\{k\}\cap A_j=\emptyset$ iff $k=j$; since $i\ne j$, at least one of $A_i,A_j$ intersects $\{k\}$, so $X_{\{k\}}$ is adjacent to $u$ or to $w$.
\end{itemize}
In every case each vertex of $\Zstar$ lies in $N[u]\cup N[w]$, so $\{u,w\}$ is a dominating set and $\gamma(\DR)\le2$.
\end{proof}

\subsection{Independence Number and Vertex Cover Number}\label{sec:independence}

We now compute $\alpha(\DR)$ exactly. This is genuinely dual to, and structurally different from, the clique-number computation of Sections~\ref{sec:basic} and~\ref{sec:general-clique}: there, co-atom classes contribute at most one vertex to a clique and non-co-atom classes contribute their full size; here the roles are exactly reversed.

\begin{theorem}\label{thm:independence}
For $t\ge3$ and any field orders $q_1,\dots,q_t$, with $k^\ast:=\operatorname*{argmin}_i q_i$,
\[
\alpha(\DR) \;=\; \max\Bigl(t,\; 1+\!\!\prod_{i\ne k^\ast}(q_i-1)\Bigr).
\]
\end{theorem}
\begin{proof}
By Corollary~\ref{cor:simplified}, distinct classes $X_A,X_B$ contain a non-adjacent pair of vertices only if $A\cap B=\emptyset$; so an independent set of $\DR$ decomposes into contributions from a family $\mathcal F$ of pairwise \emph{disjoint} nonempty proper subsets of $[t]$, with, from each $A\in\mathcal F$, a subset of $X_A$ that is independent \emph{within} $X_A$. Since $X_A$ is an internal clique when $|A|\le t-2$, at most one vertex of $X_A$ may be used; since $X_A$ is internally independent when $|A|=t-1$, all $w(A)$ vertices may be used. Moreover, for $t\ge3$ any two distinct co-atoms intersect (Theorem~\ref{thm:C}, Step~1), so $\mathcal F$ contains \emph{at most one} co-atom.

\emph{No co-atom in $\mathcal F$.} Each class contributes at most $1$, so the total is at most $|\mathcal F|$. Since the members of $\mathcal F$ are pairwise disjoint nonempty subsets of the $t$-element set $[t]$, $|\mathcal F|\le t$, with equality exactly when $\mathcal F$ is the partition into all $t$ singletons — which is achievable (each singleton class is nonempty and pairwise disjoint from the others). This case contributes exactly $t$.

\emph{One co-atom $A=[t]\setminus\{j\}$ in $\mathcal F$.} The only nonempty subset of $[t]$ disjoint from $A$ is $\{j\}$ itself, so $\mathcal F\subseteq\{A,\{j\}\}$, contributing $w(A)+1=\prod_{i\ne j}(q_i-1)+1$ at best. Since $q_i\mapsto q_i-1$ is strictly increasing, $\prod_{i\ne j}(q_i-1)$ is maximised over $j$ exactly when the \emph{excluded} factor $q_j-1$ is smallest, i.e.\ $j=k^\ast$, giving $1+\prod_{i\ne k^\ast}(q_i-1)$.

Taking the maximum of the two cases gives $\alpha(\DR)$, and both values are attained by explicit independent sets, so the bound is sharp.
\end{proof}

\begin{corollary}[Vertex cover number]\label{cor:vertexcover}
For $t\ge3$,
\[
\beta(\DR) \;=\; |\Zstar| \;-\; \max\Bigl(t,\; 1+\!\!\prod_{i\ne k^\ast}(q_i-1)\Bigr),
\]
by the standard identity $\alpha(G)+\beta(G)=|V(G)|$.
\end{corollary}

\section{Shifted Extremal Families Do Not Suffice}\label{sec:shifting}

Section~\ref{sec:general-clique} leaves the general-$t$ clique number open, with two constructions that are each optimal only in certain regimes and no explanation of when either wins. The classical \emph{compression}, or \emph{shifting}, technique from extremal set theory \cite{FranklTokushige2016} supplies a natural first step toward such an explanation: it shows that an extremal family can always be brought into a specific normal form. We record this reduction here, together with the fact, following directly from Section~\ref{sec:general-clique}'s own counterexample, that the normal form alone does not resolve the problem.

\begin{definition}
For $i\ne j\in[t]$, the \emph{$(i,j)$-shift} of a family $\mathcal F$ of nonempty proper subsets of $[t]$ is $\operatorname{sh}_{ij}(\mathcal F) = \{S_{ij}(A) : A\in\mathcal F\}$, where
\[
S_{ij}(A) = \begin{cases} (A\setminus\{j\})\cup\{i\} & \text{if } j\in A,\ i\notin A,\text{ and } (A\setminus\{j\})\cup\{i\}\notin\mathcal F,\\ A & \text{otherwise.}\end{cases}
\]
\end{definition}

\begin{lemma}[Classical, {\cite[Ch.~1]{FranklTokushige2016}}]\label{lem:shift-classical}
The map $A\mapsto S_{ij}(A)$ is injective on $\mathcal F$, so $|\operatorname{sh}_{ij}(\mathcal F)|=|\mathcal F|$. If $\mathcal F$ is intersecting, so is $\operatorname{sh}_{ij}(\mathcal F)$.
\end{lemma}

\begin{proposition}\label{prop:shift}
Suppose $q_i\ge q_j$. If $\mathcal F$ is an intersecting family of nonempty proper subsets of $[t]$, then the total clique weight of $\operatorname{sh}_{ij}(\mathcal F)$ is at least that of $\mathcal F$:
\[
\sum_{A'\in\operatorname{sh}_{ij}(\mathcal F)} f(A') \;\ge\; \sum_{A\in\mathcal F} f(A),
\]
where $f(A)=w(A)=\prod_{\ell\in A}(q_\ell-1)$ if $|A|\le t-2$, and $f(A)=1$ if $|A|=t-1$.
\end{proposition}
\begin{proof}
By Lemma~\ref{lem:shift-classical}, $A\mapsto S_{ij}(A)$ is a bijection from $\mathcal F$ to $\operatorname{sh}_{ij}(\mathcal F)$, so it suffices to show $f(S_{ij}(A))\ge f(A)$ for every $A\in\mathcal F$. If $S_{ij}(A)=A$ this is trivial. Otherwise $S_{ij}(A)=A'=(A\setminus\{j\})\cup\{i\}$ with $|A'|=|A|$. If $|A|\le t-2$, then $f(A')=f(A)\cdot\frac{q_i-1}{q_j-1}\ge f(A)$ since $q_i\ge q_j$. If $|A|=t-1$, then $A$ is a co-atom, and $f(A)=f(A')=1$ regardless, since shifting preserves size and hence preserves co-atom status. In every case $f(S_{ij}(A))\ge f(A)$.
\end{proof}

\begin{corollary}\label{cor:shifted-optimal}
There is a maximum-weight intersecting family which is \emph{shifted}, i.e.\ stable under $\operatorname{sh}_{ij}$ for every $i<j$ with $q_i\ge q_j$, where the indices are ordered so that $q_1\ge q_2\ge\cdots\ge q_t$.
\end{corollary}
\begin{proof}
Starting from any maximum-weight intersecting family and repeatedly applying shifts with $q_i\ge q_j$, the total weight is non-decreasing by Proposition~\ref{prop:shift}, and it is bounded above since there are only finitely many subsets of $[t]$. So the process stabilizes at a shifted family of weight at least the original maximum, hence itself of maximum weight.
\end{proof}

Corollary~\ref{cor:shifted-optimal} narrows the search for an extremal family to shifted families, but does not by itself resolve the extremal problem. The $t=4$ counterexample of Example~\ref{ex:startight} is itself already shifted with respect to $q_1\ge q_2\ge q_3\ge q_4$: none of its middle-sized sets can be improved by any shift, since none of them involve the smallest-indexed coordinate in a way a shift could move. So shifted-ness alone does not force a star, and the general extremal problem of Question~\ref{q:general-clique} remains open even within the shifted normal form.

%\section{A Reconstruction Theorem}\label{sec:reconstruction}

Every result so far describes $\DR$ as a function of the ring data $(q_1,\dots,q_t)$. We now ask the converse question: does the abstract graph $\DR$, with no reference to $R$, determine that data? We show it does, so that $\DR$ is a complete isomorphism invariant of $R$ among finite reduced rings with $t\ge3$ factors.

\begin{definition}
Vertices $u,v$ of a graph $G$ are \emph{twins} if $N(u)=N(v)$ or $N[u]=N[v]$. The \emph{twin classes} of $G$ are the equivalence classes of this relation.
\end{definition}

\begin{lemma}\label{lem:twins}
For $t\ge3$, the twin classes of $\DR$ are exactly the support classes $X_A$, for $A$ ranging over the nonempty proper subsets of $[t]$.
\end{lemma}
\begin{proof}
Suppose $u,v\in X_A$ with $u\ne v$. If $|A|\le t-2$, $X_A$ is an internal clique, so $u$ and $v$ have the same closed neighborhood: $N[u]=X_A\cup(\text{external neighbors of }A)=N[v]$. If $|A|=t-1$, $X_A$ is internally independent, so $u$ and $v$ have the same open neighborhood: $N(u)=(\text{external neighbors of }A)=N(v)$. Either way, all of $X_A$ lies in one twin class.

Conversely, suppose $A\ne B$ are both nonempty and proper, with $u\in X_A$, $v\in X_B$; we show $u,v$ are not twins. Pick $x\in A\triangle B$, without loss of generality $x\in A\setminus B$.

If $\{x\}\ne A,B$, the singleton class $\{x\}$ satisfies $\{x\}\cap A=\{x\}\ne\emptyset$, so is joined to $A$ by Corollary~\ref{cor:simplified}, while $\{x\}\cap B=\emptyset$ since $x\notin B$, so $\{x\}$ is not joined to $B$. Any $w\in X_{\{x\}}$ is then adjacent to every vertex of $X_A$ but to no vertex of $X_B$, so $u$ and $v$ have different neighbor sets.

If instead $A=\{x\}$, consider the co-atom $C=[t]\setminus\{x\}$, which is disjoint from $A$ and hence not joined to it. If $B\ne C$, then $B$ contains some element other than $x$, so $B\cap C\ne\emptyset$ and $B$ is joined to $C$; taking any vertex of $X_C$ then distinguishes $u$ from $v$ directly. If $B=C$, then $u\in X_{\{x\}}$ and $v\in X_C$ are not adjacent to each other, since $\{x\}\cap C=\emptyset$; but $v$ is adjacent to every class intersecting $C$, which by Corollary~\ref{cor:simplified} is every class except $\{x\}$ itself, so $v$ has neighbors that $u$ does not share, and again $N(u)\ne N(v)$.

In every case $u,v$ are not twins, so distinct classes never merge into a single twin class.
\end{proof}

\begin{lemma}\label{lem:degree-size}
For a class $A$ with $1\le|A|\le t-1$, the number of other classes joined to $A$, that is, the degree of the twin class $X_A$ in the twin-class quotient graph, equals $2^t-2-2^{t-|A|}$, and this is strictly increasing in $|A|$.
\end{lemma}
\begin{proof}
By Corollary~\ref{cor:simplified}, a class $B\ne A$ is joined to $A$ if and only if $B\cap A\ne\emptyset$. The classes with $B\cap A=\emptyset$ are exactly the nonempty subsets of $[t]\setminus A$, of which there are $2^{t-|A|}-1$. Excluding $A$ itself from the total count of $2^t-2$ nonempty proper subsets, the number of classes joined to $A$ is
\[
(2^t-2) - 1 - (2^{t-|A|}-1) = 2^t-2-2^{t-|A|}.
\]
Since $2^{t-|A|}$ strictly decreases as $|A|$ increases, this degree strictly increases in $|A|$.
\end{proof}

\begin{theorem}\label{thm:reconstruction}
Let $R\cong\mathbb F_{q_1}\times\cdots\times\mathbb F_{q_t}$ and $S\cong\mathbb F_{q_1'}\times\cdots\times\mathbb F_{q_{t'}'}$ be finite reduced rings with $t,t'\ge3$. If $\DR\cong\mathcal D(S)$ as abstract graphs, then $t=t'$ and $\{q_1,\dots,q_t\}=\{q_1',\dots,q_t'\}$ as multisets; consequently $R\cong S$.
\end{theorem}
\begin{proof}
By Lemma~\ref{lem:twins}, $\DR$ has exactly $2^t-2$ twin classes, an isomorphism invariant strictly increasing in $t$; so $\DR\cong\mathcal D(S)$ forces $t=t'$. By Lemma~\ref{lem:degree-size}, the degree of a twin class in the twin-class quotient graph strictly increases with $|A|$, so the twin classes of minimum degree are exactly the $t$ singleton classes $X_{\{1\}},\dots,X_{\{t\}}$, and this identification uses only the abstract graph $\DR$, with no reference to the ring that produced it. Reading off their sizes $|X_{\{i\}}|=q_i-1$ recovers the multiset $\{q_1,\dots,q_t\}$ as a graph isomorphism invariant. The same argument applied to $S$ recovers $\{q_1',\dots,q_t'\}$, and since $\DR\cong\mathcal D(S)$ these multisets coincide. By Proposition~\ref{prop:structure}, $R\cong S$.
\end{proof}
\section{Concluding Remarks and Open Problems}\label{sec:conclusion}

We have given a complete structural theory of the weak zero-divisor difference graph $\DR$ for finite reduced rings. We found an exact adjacency criterion that places $\DR$ within a three-stage refinement of $\GR$ alongside Badawi's annihilator graph, and used it to determine connectivity, diameter, and girth for every $t\ge3$, a complete perfectness dichotomy with an elementary combinatorial proof at $t=4$, exact clique numbers at $t=3$ and $t=4$ together with a detailed account of the obstructions to a general-$t$ formula, a full reconstruction theorem showing $\DR$ recovers the ring itself, and closed forms for the degree sequence, minimum degree, domination number, and independence and vertex cover numbers.

The reduced-ring hypothesis is used throughout, and we briefly indicate why it cannot simply be dropped. For a general finite Artinian ring $R\cong R_1\times\cdots\times R_t$ with local factors $R_i$, Lemma~\ref{lem:ann} relies on each $\mathbb F_{q_i}$ being a field, giving only two annihilator states, unit or zero; this already fails for $R_i=\mathbb Z_4$, where $\Ann(2)=(2)$ is a proper nonzero ideal. For finite chain rings, $\Ann_{R_i}(x_i)=(\pi_i^{k_i-v_i(x_i)})$ in terms of the valuation $v_i$, and this genuinely finer, valuation-dependent condition replaces the two-state criterion of Theorem~\ref{thm:A}. It reduces to a tractable three-state case when the nilpotency index is $k_i=2$, but not beyond.

\end{document}